\documentclass[12pt]{amsart}
\usepackage{amsmath,amsfonts,amssymb,amsthm}
\usepackage{mathrsfs}

\makeatletter
\@namedef{subjclassname@2010}{%
  \textup{2010} Mathematics Subject Classification}
\makeatother

\newtheorem{thm}{Theorem}[section]
\newtheorem{cor}[thm]{Corollary}

\newtheorem{prop}[thm]{Proposition}

\date{}

\newcommand{\dlines}{\displaylines}

\newcommand{\field}[1]{\mathbb{#1}}

\newcommand{\R}{\field{R}}

\renewcommand{\th}{\theta}

\newcommand{\X}{\mathscr{X}}

\newcommand{\bS}{\field{S}}

\newcommand{\C}{\field{C}}

\newcommand{\cS}{{\mathcal S}}

\def\tr{\mathop{\rm tr}}

\def\cS{{\field S}}

\def\E{{\mathbb{ E}}}

\def\P{{\mathbb{P}}}

\def\paref#1{(\ref{#1})}

\newtheorem{theorem}{Theorem}[section]

\newtheorem{rema}[theorem]{Remark}

\newtheorem{lemma}[theorem]{Lemma}

\newtheorem{remark}[theorem]{Remark}

\numberwithin{equation}{section}

\begin{document}

\title[On L\'evy's Brownian motion]{On L\'evy's Brownian motion indexed by the elements of
compact groups\footnote{R\lowercase{esearch supported by \uppercase{ERC} grant 277742} {\it P\lowercase{ascal}}.}}

\author[P. Baldi]{Paolo Baldi}
\address{Dipartimento di Matematica\\ Universit\`a
di Roma {\it Tor Vergata}\\
00133 Roma, Italy}
\email{baldi@mat.uniroma2.it}

\author[M. Rossi]{Maurizia Rossi}
\address{Dipartimento di Matematica\\ Universit\`a
di Roma {\it Tor Vergata}\\
00133 Roma, Italy}
\email{rossim@mat.uniroma2.it}

\begin{abstract}
\noindent We investigate positive definiteness of the Brownian kernel
$K(x,y)=\frac{1}{2} \,( d(x,x_0) + d(y,x_0) - d(x,y) )$ on a
compact group $G$ and in particular for  $G=SO(n)$.
\end{abstract}



\maketitle

\section{Introduction}
In 1959 P.L\'evy \cite{bib:L} asked the question of the existence of a process $X$ indexed by the points of
a metric space $(\X,d)$ and generalizing the Brownian motion, i.e. of a real Gaussian process which would be centered, vanishing at some point $x_0\in\X$ and such that $\E(|X_x-X_y|^2)=d(x,y)$. By polarization, the covariance function of such a process would be
\begin{equation}\label{kernel MB}
K(x,y)=\frac{1}{2} \,( d(x,x_0) + d(y,x_0) - d(x,y) )
\end{equation}
so that this question is equivalent to the fact that the kernel $K$ is positive definite. Positive
definiteness of $K$ for $\X=\R^m$ and $d$ the Euclidean metric had been proved by Schoenberg \cite{MR1501980} in 1938 and P.L\'evy itself
constructed the Brownian motion on $\X=\bS^{m-1}$, the euclidean sphere of $\R^m$,
 $d$ being the distance along the geodesics.
Later Gangolli \cite{GAN:60} gave an analytical proof of the positive definiteness of the kernel \paref{kernel MB} for the same metric space $(\bS^{m-1},d)$, in a paper that dealt with this question for a large class of homogeneous spaces.

Finally Takenaka in \cite{bib:TKU} proved the positive definiteness of the kernel \paref{kernel MB} for the Riemannian metric spaces
of constant sectional curvature equal to $-1,0$ or $1$, therefore adding the hyperbolic disk to the list. To be precise in the case of the hyperbolic space
$\mathcal{H}_m = \lbrace (x_0, x_1, \dots, x_m)\in \R^{m+1} :
x_1^2 + \dots x_m^2 - x_0^2 = 1 \rbrace $, the distance under consideration  is the unique, up to  multiplicative
constants, Riemannian distance that is invariant with respect to the action of $G=L_m$, the Lorentz group.

In this short note we investigate this question for the cases
$\X=SO(n)$. The answer is that the kernel \paref{kernel MB} is not
positive definite on $SO(n)$ for $n>2$. This is somehow surprising
as, in particular, $SO(3)$ is locally isometric to $SU(2)$, where
positive definiteness of the kernel $K$ is immediate as shown below.

We have been led to the question of the existence of the Brownian motion indexed by the elements of these groups - in particular of $SO(3)$ - in connection with the analysis
and the modeling of
the Cosmic Microwave Background
which has become recently an active research field  (see \cite{bib:LS}, \cite{bib:M}, \cite{dogiocam}, \cite{MR2867816} e.g.)
and that has attracted the attention to the study of
random fields (\cite{BM06}, \cite{MR2342708}, \cite{Pepy} e.g.).
More precisely, in the modern cosmological models the CMB is seen as the realization of an invariant random field in a vector bundle over the sphere $\cS^2$ and the analysis of its components (the polarization e.g.)
requires the \emph{spin} random fields theory.
This
leads naturally to the investigation of invariant random fields on $SO(3)$
enjoying particular properties and therefore to the question
of the existence of a privileged random field i.e. L\'evy's Brownian random field on $SO(3)$.

In \S\ref{elem} we recall some elementary facts about invariant distances and positive definite kernels. In
\S\ref{sud} we treat the case $G=SU(2)$, recalling well known facts about the invariant distance and Haar measure of this group.
Positive definiteness of $K$ for $SU(2)$ is just a simple remark, but these facts are needed in
\S\ref{sot} where we treat the case $SO(3)$ and deduce from the case $SO(n)$, $n\ge 3$.
\section{Some elementary facts}\label{elem}
In this section we recall some well known facts about Lie groups (see mainly
\cite{faraut} and also \cite{MR2088027}, \cite{sugiura}).
\subsection{Invariant distance of a compact Lie group}
From now on we denote by $G$ a compact Lie group. It is well known that  $G$ admits {at least} a bi-invariant Riemannian metric
(see \cite{MR2088027} p.66 e.g.), that
we shall denote $\lbrace \langle \cdot, \cdot \rangle_g \rbrace_{g\in G}$ where of course
$\langle \cdot, \cdot \rangle_g$ is the inner product defined on the tangent space $T_g G$ to the manifold $G$ at $g$ and the family $\lbrace \langle \cdot, \cdot \rangle_g \rbrace_{g\in G}$ smoothly depends on $g$. By the bi-invariance property, for $g\in G$ the diffeomorphisms
$L_g$ and $R_g$ (resp. the left multiplication and the right multiplication of the group) are isometries.
Since the tangent space $T_g G$ at any point $g$ can be translated to the tangent space $T_e G$ at the identity element $e$ of the group, the metric $\lbrace \langle \cdot, \cdot \rangle_g \rbrace_{g\in G}$ is completely characterized by $\langle \cdot, \cdot \rangle_e$.
Moreover, $T_e G$ being the Lie algebra $\mathfrak g$ of $G$, the bi-invariant metric corresponds to an inner product $\langle \cdot, \cdot \rangle$ on
$\mathfrak g$ which is invariant under the adjoint representation $Ad$ of $G$. Indeed there is a one-to-one correspondence between bi-invariant Riemannian metrics on $G$ and $Ad$-invariant inner products on
$\mathfrak g$. If in addition $\mathfrak g$ is {semisimple}, then the negative Killing form
of $G$ is an $Ad$-invariant inner product on $\mathfrak g$ itself.

If there exists a unique
(up to a multiplicative factor) bi-invariant metric on $G$ (for a sufficient condition see
\cite{MR2088027}, Th. $2.43$) and $\mathfrak g$ is semisimple, then
this metric is necessarily proportional to the negative Killing form of $\mathfrak g$. It is well known that this is the case for $SO(n), (n\ne 4)$ and $SU(n)$; furthermore, the (natural) Riemannian metric on $SO(n)$ induced by the embedding
$SO(n) \hookrightarrow \R^{n^2}$ corresponds to the negative Killing form of ${so}(n)$.

Endowed with this bi-invariant Riemannian metric, $G$ becomes a {metric space}, with a distance $d$ which is
bi-invariant. Therefore the function $g\in G \to d(g,e)$ is a class function as
\begin{equation}\label{cfunction}
d(g,e)=d(hg,h)=d(hgh^{-1},hh^{-1})=d(hgh^{-1},e), \qquad g,h\in G\ .
\end{equation}
It is well known that {geodesics} on $G$ through the identity $e$
are exactly the one parameter subgroups of $G$ (see \cite{MR0163331} p.113 e.g.), thus a geodesic from $e$
is the curve on $G$
\begin{equation*}
\gamma_X(t) : t\in [0,1] \to \exp(tX)
\end{equation*}
for some $X\in \mathfrak g$. The length of  this geodesic is
\begin{equation*}
L(\gamma_X) = \| X \| = \sqrt{\langle X, X \rangle}\ .
\end{equation*}
Therefore
\begin{equation*}\label{distanza}
d(g,e) = \inf_{X\in \mathfrak g: \exp X=g} \| X \|\ .
\end{equation*}

\subsection{Brownian kernels on a metric space}

Let $(\X, d)$ be a metric space.
\begin{lemma}\label{fondamentale}
The kernel $K$ in (\ref{kernel MB}) is positive definite on $\X$ if
and only if  $d$ is a restricted negative definite kernel, i.e., for
every choice of elements $x_1,  \dots, x_n\in \X$ and of complex
numbers $\xi_1, \dots, \xi_n$ with $\sum_{i=1}^n \xi_i =0$
\begin{equation}\label{neg def}
\sum_{i,j=1}^n d(x_i,x_j) \xi_i \overline{\xi_j} \le 0\ .
\end{equation}
\end{lemma}

\begin{proof}
For every $x_1, \dots, x_n\in \X$ and complex numbers $\xi_1,\dots,
\xi_n$
\begin{equation}\label{eq1}
\sum_{i,j} K(x_i,x_j) \xi_i \overline{\xi_j} = \frac{1}{2} \Bigl(
\overline{a} \sum_i d(x_i, x_0)\xi_i  + a \sum_j d(x_j,
x_0)\overline{\xi_j}- \sum_{i,j}d(x_i, x_j)\xi_i \overline{\xi_j}
\Bigr)
\end{equation}
where $a:= \sum_i \xi_i$.  If  $a=0$ then it is immediate that in (\ref{eq1}) the l.h.s. is $\ge 0$ if and only if the r.h.s. is $\le 0$. Otherwise set
$\xi_{0}:=-a$ so that $\sum_{i=0}^n \xi_i =0$. The following equality
\begin{equation}
\sum_{i,j=0}^{n} K(x_i,x_j) \xi_i \overline{\xi_j} =
\sum_{i,j=1}^n K(x_i,x_j) \xi_i \overline{\xi_j}
\end{equation}
is then easy to check, keeping in mind that $K(x_i,x_0)=K(x_0,
x_j)=0$, which finishes the proof.
\end{proof}
For a more general proof see \cite{GAN:60} p. $127$ in the proof of Lemma 2.5.

If $\X$ is the homogeneous space of some topological group $G$, and
$d$ is a $G$-invariant distance, then (\ref{neg def}) is satisfied
if and only if for every choice of elements $g_1,\dots,g_n\in G$ and
of complex numbers $\xi_1,\dots, \xi_n$ with $\sum_{i=1}^n \xi_i =0$
\begin{equation}\label{neg def2}
\sum_{i,j=1}^n d(g_ig_j^{-1}x_0,x_0) \xi_i \overline{\xi_j} \le 0
\end{equation}
where $x_0\in \X$ is a fixed point. We shall say that
the function $g\in G \to d(gx_0, x_0)$ is restricted negative definite on $G$ if it satisfies (\ref{neg def2}).

In our case of interest $\X=G$  a compact  (Lie) group  and
$d$ is a bi-invariant distance as in $\S 2.1$. The Peter-Weyl development (see \cite{faraut} e.g.) for the class function $d(\cdot,e)$ on $G$ is
\begin{equation}\label{PW dev}
d(g,e)= \sum_{\ell \in \widehat G} \alpha_\ell \chi_\ell(g)
\end{equation}
where $\widehat G$ denotes the family of equivalence classes of irreducible representations of $G$ and $\chi_\ell$  the character of the
$\ell$-th irreducible representation of $G$.

\begin{remark}\label{coeff neg}\rm
A function $\phi$ with a development as in (\ref{PW dev}) is
restricted negative definite if and only if $\alpha_\ell \le 0$ but
for the trivial representation.

Actually note first that, by standard approximation arguments,
$\phi$ is restricted negative definite if and only if for every
continuous function $f:G\to \C$ with $0$-mean (i.e. orthogonal to
the constants)
\begin{equation}\label{neg def measure}
\int_G\int_G \phi(gh^{-1}) f(g)\overline{f(h)}\,dg\, dh \le 0
\end{equation}
$dg$ denoting the Haar measure of $G$.
Choosing $f=\chi_\ell$ in the l.h.s. of
(\ref{neg def measure}) and denoting $d_\ell$ the dimension of the corresponding representation, a straightforward computation gives
\begin{equation}\label{semplice}
\int_G\int_G \phi(gh^{-1}) \chi_\ell(g)\overline{\chi_\ell(h)}\,dg\,
dh = \frac{\alpha_\ell}{d_\ell}
\end{equation}
so that if $\phi$ restricted negative definite, $\alpha_\ell\le 0$
necessarily.

Conversely, if $\alpha_\ell \le 0$ 
but for the
trivial representation, then $\phi$ is restricted negative definite,
as the characters $\chi_\ell$'s are positive definite and orthogonal
to the constants.
\end{remark}

\section{$SU(2)$}\label{sud}

The special unitary group $SU(2)$ consists of the complex unitary $2\times 2$-matrices $g$ such that
$\det(g)=1$.
Every $g\in SU(2)$ has the form
\begin{equation}\label{matrice}
g= \begin{pmatrix} a  & b \\
 -\overline{b} & \overline{a}
\end{pmatrix}, \qquad a,b\in \C,\, |a|^2 + |b|^2 = 1\ .
\end{equation}
If $a=a_1 + ia_2$ and $b=b_1 + ib_2$, then the map
\begin{align}\label{omeomorfismo}
\Phi(g)=
(a_1, a_2, b_1, b_2)
\end{align}
is an {homeomorphism} (see \cite{faraut}, \cite{sugiura} e.g.) between $SU(2)$ and the unit sphere $\cS^3$
of $\R^4$. Moreover the right translation
\begin{equation*}
R_g : h\to hg, \qquad h,g\in SU(2)
\end{equation*}
of $SU(2)$ is a rotation (an element of $SO(4)$) of $\cS^3$ (identified with $SU(2)$).
The homeomorphism (\ref{omeomorfismo}) preserves the invariant measure, i.e., if $dg$ is the normalized Haar measure
on $SU(2)$, then $\Phi(dg)$ is the normalized Lebesgue measure on $\cS^3$.
As the $3$-dimensional polar coordinates on $\cS^3$ are
\begin{equation}\label{polar coord}
\begin{array}{l}
a_1=\cos \theta,\cr
a_2= \sin \theta\,\cos \varphi,\cr
b_1= \sin \theta\,\sin \varphi\,\cos \psi,\cr
b_2= \sin \theta\,\sin \varphi\,\sin \psi\ ,
\end{array}
\end{equation}
$(\theta, \varphi, \psi) \in [0,\pi] \times [0,\pi]\times [0,2\pi]$, the normalized Haar integral of $SU(2)$ for an integrable function $f$ is
\begin{equation}\label{int}
\int_{SU(2)} f(g)\,dg = \frac{1}{2\pi^2}\int_0^\pi\sin \varphi\,  d\varphi\,
\int_0^\pi \sin^2 \theta\,d\theta\, \int_0^{2\pi}
f(\theta, \varphi, \psi)\, d\psi
\end{equation}
The bi-invariant Riemannian metric on $SU(2)$ is necessarily proportional to the negative
Killing form  of its Lie algebra ${su(2)}$  (the real vector space of
$2\times 2$ anti-hermitian complex matrices).
We consider the bi-invariant metric corresponding to the $Ad$-invariant inner product on ${su(2)}$
\begin{equation*}
\langle X, Y \rangle= -\frac12\,{\tr(XY)},\qquad X,Y \in {su(2)}\ .
\end{equation*}
Therefore as an orthonormal basis of ${su(2)}$ we can consider the matrices
$$
\dlines{
X_1= \begin{pmatrix} 0  & 1 \\
 -1 & 0
\end{pmatrix}, \quad
X_2= \begin{pmatrix} 0  & i \\
 i & 0
\end{pmatrix}, \quad
X_3 = \begin{pmatrix} i  & 0\\
 0 & -i
\end{pmatrix}
}$$
The homeomorphism (\ref{omeomorfismo}) is actually an isometry between $SU(2)$ endowed with this
distance and $\cS^3$. Hence the restricted negative definiteness of the kernel $d$ on $SU(2)$ is an immediate consequence of this property on $\cS^3$ which is known
to be true as mentioned in the introduction (\cite{GAN:60}, \cite{bib:L}, \cite{bib:TKU}).
In order to develop a comparison with $SO(3)$, we shall give a different proof of this fact in \S\ref{s-final}.
\section{$SO(n)$}\label{sot}

We first investigate the case $n=3$. The group
$SO(3)$ can also be realized as a quotient of $SU(2)$. Actually
the adjoint representation $Ad$ of $SU(2)$
is a surjective morphism from $SU(2)$ onto $SO(3)$ with kernel $\lbrace \pm e \rbrace$ (see \cite{faraut} e.g.).
Hence the well known result
\begin{equation}\label{iso}
SO(3) \cong {SU(2)}/{\lbrace \pm e \rbrace}\ .
\end{equation}
Let us explicitly recall this morphism: 
if $a=a_1 +ia_2, b=b_1 + ib_2$ with $|a|^2 + |b|^2=1$ and
$$
\widetilde g=\begin{pmatrix} a  & b \\
 -\overline{b} & \overline{a}
\end{pmatrix}$$
then the orthogonal matrix $Ad(\widetilde g)$  is given by
\begin{equation}\label{matr}
g=\begin{pmatrix}
a_1^2-a_2^2-(b_1^2-b_2^2)&-2a_1a_2-2b_1b_2&-2(a_1b_1-a_2b_2)\cr
2a_1a_2-2b_1b_2&(a_1^2-a_2^2)+(b_1^2-b_2^2)&-2(a_1b_2+a_2b_1)\cr
2(a_1b_1+a_2b_2)&-2(-a_1b_2+a_2b_1)&|a|^2-|b|^2
\end{pmatrix}
\end{equation}
The isomorphism in (\ref{iso}) might suggest that the
positive definiteness of the Brownian kernel on $SU(2)$ implies
a similar result for $SO(3)$.
This is not true and actually it turns out that the distance $(g,h) \to d(g,h)$ on $SO(3)$
induced by its bi-invariant Riemannian metric \emph{is not
a restricted negative definite kernel} (see Lemma \ref{fondamentale}).

As for $SU(2)$, the bi-invariant Riemannian metric on $SO(3)$ is proportional to the negative Killing form  of its Lie algebra
${ so(3)}$ (the real $3\times 3$ antisymmetric real matrices).
We shall consider the $Ad$-invariant inner product on ${so(3)}$ defined as
\begin{equation*}
\langle A,B \rangle = -\frac12\,{\tr(AB)}\ , \qquad A,B\in {so(3)}\ .
\end{equation*}
An  orthonormal basis for ${so(3)}$ is therefore given by
the matrices
$$\dlines{
A_1= \begin{pmatrix} 0  & 0 & 0\\
 0 & 0 & -1\\
 0 & 1 & 0
\end{pmatrix}, \quad
A_2= \begin{pmatrix} 0  & 0 & 1 \\
 0 & 0 & 0\\
 -1 & 0 & 0
\end{pmatrix}, \quad
A_3 = \begin{pmatrix} 0 & -1  & 0\\
 1 & 0 & 0\\
 0 & 0 & 0
\end{pmatrix}
}$$
Similarly to the case of $SU(2)$, it is easy to compute the distance
from  $g\in SO(3)$ to the identity. Actually $g$ is conjugated to the matrix of the form
\begin{equation*}
\Delta(t)= \begin{pmatrix} \cos t  & \sin t & 0 \\
 -\sin t & \cos t & 0\\
 0 & 0 & 1
\end{pmatrix} = \exp(tA_1)
\end{equation*}
where $t\in [0,\pi]$ is the {rotation angle} of $g$.
Therefore if $d$ still denotes the distance induced by the bi-invariant metric,
\begin{equation*}
d(g,e) = d( \Delta(t), e ) = t
\end{equation*}
i.e. the distance from $g$ to $e$ is the rotation angle of $g$.

Let us denote $\lbrace \chi_\ell \rbrace_{\ell \ge 0}$ the set of characters for $SO(3)$.
It is easy to compute the Peter-Weyl development in (\ref{PW dev}) for $d(\cdot, e)$  as  the
characters $\chi_\ell$ are also simple functions of the rotation angle.
More precisely, if $t$ is the rotation angle of $g$ (see \cite{dogiocam} e.g.),
\begin{equation*}
\chi_\ell(g)= \frac{\sin\frac{(2\ell + 1)t}{2}}{\sin\frac{t}{2}}=1 + 2\sum_{m=1}^\ell \cos(mt)\ .
\end{equation*}
We shall prove that the coefficient
$$
\alpha_\ell=\int_{SO(3)} d(g,e)\chi_\ell(g)\, dg
$$
is positive for some $\ell\ge 1$. As both $d(\cdot,e)$ and $\chi_\ell$ are functions of the rotation angle $t$,
we have
$$
\alpha_\ell=\int_0^\pi t\Bigl( 1 + 2\sum_{j=1}^\ell \cos(jt)\Bigr)\, p_T(t)\, dt
$$
where $p_T$ is the density of $t=t(g)$, considered as a r.v. on the probability space $(SO(3),dg)$. The next statements are devoted to the computation of the density $p_T$. This is certainly well
known but we were unable to find a reference in the literature. We first compute the density of the trace of $g$.
\begin{prop}
The distribution of the trace of a matrix in $SO(3)$ with respect to the normalized Haar measure is given by the density
\begin{equation}\label{trace3}
f(y)=\frac 1{2\pi}\,(3-y)^{1/2}(y+1)^{-1/2}1_{[-1,3]}(y)\ .
\end{equation}
\end{prop}
\begin{proof}
The trace of the matrix \paref{matr} is equal to
$$
\tr(g)=3a_1^2-a_2^2-b_1^2-b_2^2\ .
$$
Under the normalized Haar measure of $SU(2)$ the vector $(a_1,a_2,b_1,b_2)$ is uniformly distributed on the sphere $\cS^3$.
Recall the normalized Haar integral (\ref{int})
so that, taking the corresponding marginal, $\th$ has density
\begin{equation}\label{denth}
f_1(\th)=\frac 2\pi\,\sin^2(\th)\, d\th\ .
\end{equation}
Now
$$
\dlines{
3a_1^2-a_2^2-b_1^2-b_2^2=4\cos^2\th-1\ .\cr
}
$$
Let us first compute the density of $Y=\cos^2 X$, where
$X$ is distributed according to the density \paref{denth}. This is elementary as
$$
\dlines{
F_Y(t)=\P(\cos^2 X\le t)=
\P(\arccos(\sqrt{t})\le X\le \arccos(-\sqrt{t}))
=\frac 2\pi\!\int\limits_{\arccos(\sqrt{t})}^{\arccos(-\sqrt{t})}
\sin^2(\th)\, d\th\ .\cr
}
$$
Taking the derivative it is easily found that the density of $Y$ is, for $0<t<1$,
$$
\dlines{
F'_Y(t)=\frac 2\pi\,(1-t)^{1/2}t^{-1/2}\ .\cr
}
$$
By an elementary change of
variable the distribution of the trace $4Y-1$ is therefore given by \paref{trace3}.

\end{proof}
\begin{cor}
The distribution of the rotation angle of a matrix in $SO(3)$ is
\begin{equation*}
p_T(t)=\frac1\pi\,(1 - \cos t)\,1_{[0,\pi]}(t)\ .
\end{equation*}
\end{cor}
\begin{proof}
It suffices to remark that if $t$ is the rotation angle of $g$, then its trace is equal to $2\cos t + 1$. $p_T$ is therefore the distribution of $W=\arccos ( \frac{Y-1}{2})$, $Y$ being distributed as \paref{trace3}. The elementary details are left to the reader.

\end{proof}
Now it is easy to compute the Fourier development of the function $d(\cdot, e)$.
\begin{prop}\label{kernel MB su SO(3)}
The kernel $d$ on $SO(3)$ is not restricted negative definite.
\end{prop}
\begin{proof}
It is enough to show that in the Fourier development
$$
d(g,e)=\sum_{\ell \ge 0} \alpha_\ell \chi_\ell(g)
$$
$\alpha_\ell > 0$ for some $\ell \ge 1$ (see Remark \ref{coeff neg}).
We have
$$
\dlines{ \alpha_\ell =\int_{SO(3)} d(g,e) \chi_\ell(g) dg = \frac 1\pi
\int_0^\pi t \Bigl( 1 + 2\sum_{m=1}^\ell \cos(mt) \Bigr) (1-\cos t)\, dt =\cr
=\frac{1}{\pi}\underbrace{\int_0^\pi t (1-\cos t)\, dt}_{:= I_1} + \frac{2}{\pi}\sum_{m=1}^{\ell} \underbrace{\int_0^\pi t \cos(mt)\, dt}_{:= I_2}
- \frac{2}{\pi}\sum_{m=1}^{\ell} \underbrace{\int_0^\pi t \cos(mt)\cos t\, dt}_{:= I_3}\ .\cr
}$$
Now integration by parts gives
$$
I_1 = \frac{\pi^2}{2} + 2,\quad
I_2= 
\frac{ (-1)^m -1}{m^2}\ \raise2pt\hbox{,}
$$
whereas, if $m\ne 1$, we have
$$
\dlines{I_3=\int_0^\pi t \cos(mt)\cos t\, dt 
= \frac{ m^2 +1}{ (m^2 -1)^2} ( (-1)^m +1)}
$$
and for $m=1$,
$$
I_3=\int_0^\pi t\cos^2 t\, dt=\frac {\pi^2}4\ .
$$
Putting things together we find
$$
\alpha_\ell =\frac{2}{\pi} \Bigl( 1 + \sum_{m=1}^{\ell} \frac{ (-1)^m -1}{m^2} + \sum_{m=2}^{\ell} \frac{ m^2 +1}{ (m^2 -1)^2} ( (-1)^m +1) \Bigr)\ .
$$
If $\ell=2$, for instance, we find $\alpha_2=\frac{2}{9\pi}>0$, but it is easy to see
that $\alpha_\ell>0$ for every $\ell$ even.

\end{proof}

Consider now the case $n>3$. $SO(n)$ contains a closed subgroup $H$ that is isomorphic to $SO(3)$ and
the restriction to $H$ of any bi-invariant distance $d$ on $SO(n)$ is a bi-invariant distance $\widetilde d$ on $SO(3)$. 
By Proposition \ref{kernel MB su SO(3)},  $\widetilde d$ is not
restricted negative definite, therefore there exist $g_1, g_2, \dots, g_m\in H$,
$\xi_1, \xi_2, \dots, \xi_m \in \R$ with $\sum_{i=1}^m \xi_i =0$ such that
\begin{equation}
\sum_{i,j} d(g_i, g_j) \xi_i \xi_j =\sum_{i,j} \widetilde d(g_i, g_j) \xi_i \xi_j > 0\ .
\end{equation}
We have therefore 
\begin{cor}\label{kernel MB su SO(n)}
Any bi-invariant distance $d$ on $SO(n), n\ge 3$ is not
a restricted negative definite kernel.
\end{cor}
Remark that  the bi-invariant Riemannian metric on $SO(4)$ is not unique, meaning that it is not necessarily
proportional to the negative Killing form of $so(4)$.
In this case Corollary \ref{kernel MB su SO(n)} states that every such
bi-invariant distance cannot be restricted negative definite.
\section{Final remarks}\label{s-final}

We were intrigued by the different behavior of the invariant distance of $SU(2)$ and $SO(3)$ despite these groups
are locally isometric and decided to compute also for $SU(2)$ the development
\begin{equation}\label{sviluppo}
d(g,e) = \sum_{\ell} \alpha_\ell \chi_\ell(g)\ .
\end{equation}
This is not difficult as, denoting by $t$ the distance of $g$ from $e$, the characters of $SU(2)$ are
$$
\chi_\ell(g)=\frac{\sin((\ell+1)t)}{\sin t},\quad t\not=k\pi
$$
and $\chi_\ell(e)=\ell+1$ if $t=0$, $\chi_\ell(-)=(-1)^\ell(\ell+1)$ if $t=\pi$. Then it is elementary to compute, for $\ell>0$,
$$
\alpha_\ell =\frac 1\pi\,\int_0^\pi t\sin((\ell+1)t)\sin t\, dt=\begin{cases}
-\frac 8\pi\,\frac{m+1}{m^2(m+2)^2}&\ell\mbox{ odd}\cr
0&\ell\mbox{ even}
\end{cases}
$$
thus confirming the restricted negative definiteness of $d$ (see Remark \ref{coeff neg}). Remark also that the coefficients corresponding to the even numbered
representations, that are also representations of $SO(3)$, here vanish.

\subsection*{Acknowledgements}The authors wish to thank A.Iannuzzi and S.Trapani for
valuable assistance.

\bibliography{bibbase}
\bibliographystyle{amsplain}
\end{document}